\documentclass[reqno,11pt,twoside]{amsart}
\usepackage{xypic,hyperref,todonotes,comment}
\usepackage{graphics,amssymb}
\usepackage{latexsym}
\usepackage{mathrsfs}
\xyoption{curve}

\newtheorem{lemma}{Lemma}[section]
\newtheorem{proposition}[lemma]{Proposition}
\newtheorem{theorem}[lemma]{Theorem}
\newtheorem{corollary}[lemma]{Corollary}

\newtheorem*{theoremA}{Theorem I}
\newtheorem*{corollaryA}{Corollary I}
\newtheorem*{theoremB}{Theorem II}
\newtheorem*{corollaryB}{Corollary II}

\theoremstyle{definition}

\newtheorem{definition}[lemma]{Definition}
\newtheorem{remark}[lemma]{Remark}
\newtheorem{question}[lemma]{Question}

\newcommand{\Hom}{\mathrm{Hom}}
\newcommand{\End}{\mathrm{End}}
\newcommand{\Aut}{\mathrm{Aut}}
\newcommand{\ox}{\otimes}
\newcommand{\qexp}{\mathrm{qexp}}
\newcommand{\upiv}{\mathrm{piv}}
\newcommand{\Npiv}{N\mathrm{piv}}
\newcommand{\Epiv}{E\mathrm{piv}}

\renewcommand{\mod}{\mbox{-\bf{mod}}}
\newcommand{\jj}{{j^{\mathrm{Vec}}}}

\def\1{\mathbf{1}}

\def\b{\beta}

\def\e{\epsilon}

\def\g{\gamma}

\def\s{\sigma}
\def\vp{\varphi}

\def\l{\lambda}

\def\k{\Bbbk}
\def\bm1{\mathbbm{1}}
\def\bV{\mathbf{V}}

\def\BQ{{\mathbb{Q}}}

\def\Id{{\operatorname{Id}}}
\def\BC{{\mathbb{C}}}

\def\BN{{\mathbb{N}}}

\def\nuk{\nu^{\mathrm{KMN}}}
\def\nus{\nu^{\mathrm{Sh}}}

\def\Hom{{\mbox{\rm Hom}}}
\def\End{{\mbox{\rm End}}}
\def\Aut{{\mbox{\rm Aut}}}

\def\tt{\tilde{t}}
\def\RH{\mathbf{R}_H}
\def\RK{\mathbf{R}_K}

\newcommand\C[1]{{#1\mbox{-\bf{mod}}}}
\def\Id{\operatorname{Id}}

\newcommand\Rep{\operatorname{Rep}}

\numberwithin{equation}{section}

\newcommand\ord{\operatorname{ord}}

\newcommand\ol[1]{\overline{#1}}
\newcommand\replace[1]{}
\newcommand\ptr{\operatorname{\underline{ptr}}}

\newcommand\id{\operatorname{id}}
\renewcommand\o{\otimes}
\renewcommand\L{\Lambda}

\newcommand\Tr{\operatorname{Tr}}

\newcommand\inv{^{-1}}
\DeclareMathOperator\ev{{\operatorname{ev}}}
\DeclareMathOperator\coev{{\operatorname{coev}}}

\DeclareMathOperator\piv{\operatorname{piv}}
\def\piv{{\operatorname{piv}}}
\newcommand\CC{\mathscr{C}}
\newcommand\CCp{\mathscr{C}^\piv}
\newcommand\CCNp{\mathscr{C}^{N\piv}}
\newcommand\CCEp{\mathscr{C}^{E\piv}}

\newcommand\DD{\mathscr{D}}
\newcommand\DDp{\mathscr{D}^\piv}
\newcommand\FF{\mathcal F}

\def\Vec{\operatorname{Vec}}

\newcommand\du{^{\vee}}

\newcommand\op{{\operatorname{op}}}

\title[]{Gauge invariants from the powers of antipodes}
\date{}

\author{Cris Negron}
\email{negronc@lsu.edu}
\thanks{The first author was supported by NSF Postdoctoral Fellowship DMS-1503147}
\address{Department of Mathematics\\Massachusetts Institute of Technology\\
Cambridge, MA 02139, USA}
\author{Siu-Hung Ng}
\thanks{The second author was partially supported by NSF grant DMS-1501179}
\email{rng@math.lsu.edu}
\address{Department of Mathematics, Louisiana State University, Baton Rouge, LA 70803, USA}

\begin{document}

\begin{abstract}
We prove that the trace of the $n$th power of the antipode of a Hopf algebra with the Chevalley property is a gauge invariant, for each integer $n$.  As a consequence, the order of the antipode, and its square, are invariant under Drinfeld twists.  The invariance of the order of the antipode is closely related to a question of Shimizu on the pivotal covers of finite tensor categories, which we affirmatively answer for representation categories of Hopf algebras with the Chevalley property.
\end{abstract}
\thanks{}
\maketitle

\section{Introduction}

The present paper is dedicated to a study of the traces of the powers of the antipode of a Hopf algebra, and an approach to the Frobenius-Schur indicators of nonsemisimple Hopf algebras.
\par
The antipode of a Hopf algebra has emerged as an object of importance in the study of Hopf algebras. It has been proved by Radford that the order of the antipode $S$ of any finite-dimensional Hopf algebra $H$ is finite \cite{Radf76}. Moreover, the trace of $S^2$ is nonzero if, and only if, $H$ is semisimple and cosemisimple \cite{LR88a}. If the base field $\k$ is of characteristic zero, $\Tr(S^2) = \dim H$ or $0$, which characterizes respectively whether $H$  is semisimple or nonsemisimple \cite{LR88}. This means semisimplicity of $H$ is characterized by the value of $\Tr(S^2)$. In particular, $\Tr(S^2)$ is an invariant of the finite tensor category $\C{H}$. The invariance of $\Tr(S^2)$ and $\Tr(S)$ can also be obtained in any characteristic via Frobenius-Schur indicators.
\par

A generalized notion of the $n$th Frobenius-Schur (FS-)indicator $\nuk_n(H)$ has been introduced in~\cite{KMN} for studying finite-dimensional Hopf algebras $H$, which are not necessarily semisimple or \emph{pivotal}. However, $\nuk_n(H)$ coincides with the $n$th FS-indicator of the regular representation of $H$ when $H$ is semisimple, defined in \cite{LM00}. These indicators are invariants of the finite tensor categories $\C{H}$. In particular,  $\nuk_2(H)=\Tr(S)$ and $\nuk_0(H)=\Tr(S^2)$ (cf. \cite{Shim15}) are invariants of $\C{H}$.
\par

The invariance of $\Tr(S)$ and $\Tr(S^2)$ alludes to the following question to be investigated in this paper:

\begin{question} \label{q1}
For any finite-dimensional Hopf algebra $H$ with the antipode $S$, is the sequence $\{\Tr(S^{n})\}_{n \in \BN}$ an invariant of the finite tensor category $\C{H}$?
\end{question}

For the purposes of this paper, we will always assume $\k$ to be an algebraic closed field of characteristic zero, and all Hopf algebras are finite-dimensional over $\k$.

Recall that a finite-dimensional Hopf algebra $H$ has the Chevalley property if its Jacobson radical is a Hopf ideal.  Equivalently, $H$ has the Chevalley property if the full subcategory of sums of irreducible modules in $H\mod$ forms a tensor subcategory.  We provide a positive answer to Question \ref{q1} for Hopf algebras with the Chevalley property.

\begin{theoremA}[Thm.~\ref{t:2}]
Let $H$ and $K$ be finite-dimensional Hopf algebras over $\k$ with antipodes $S_H$ and $S_K$  respectively. Suppose $H$ has the Chevalley property and that $\C{H}$ and $\C{K}$ are equivalent as tensor categories. Then we have
\[
\Tr(S^n_H) = \Tr(S^n_K)
\]
for all integers $n$.
\end{theoremA}

In a categorial language, the theorem tells us that for any finite tensor category $\mathscr{C}$ with the Chevalley property which admits a fiber functor to the category of vector spaces, the ``traces of the powers of the antipode" are well defined invariants which are independent of the choice of fiber functor.  One naturally asks whether these scalars can be expressed purely in terms of categorial data of $\mathscr{C}$.
\par

Etingof has asked the question whether, for any finite-dimensional $H$, $\Tr(S^{2m})=0$ provided $\ord(S^2)\nmid m$~\cite[p186]{RS02}.  This question is affirmatively answered for pointed and dual pointed Hopf algebras in \cite{RS02}.  However, the odd powers of the antipode may have nonzero traces in general. We note that the above result covers both the even and odd powers of the antipode.
\par

 Theorem I also implies that the orders of the first two powers of the antipode of a Hopf algebra with the the Chevalley property are also invariants.

\begin{corollaryA}[Cor.~\ref{cor:ordpres}]
Let $H$ and $K$ be finite-dimensional Hopf algebras over $\k$ with antipodes $S_H$ and $S_K$  respectively. Suppose $H$ has the Chevalley property and that $\C{H}$ and $\C{K}$ are equivalent as tensor categories. Then $\ord(S_H)=\ord(S_K)$ and hence $\ord(S^2_H)=\ord(S^2_K)$.
\end{corollaryA}

The order of $S^2$ is related to a known invariant called the {\it quasi-exponent} $\mathrm{qexp}(H)$  \cite{EGqexp02}.  Namely, $\mathrm{ord}(S^2)$ divides $\mathrm{qexp}(H)$ for any finite-dimensional Hopf algebra.  However, we still do not know whether or not the order of $S^2$ is an invariant in general.
\par

The questions under consideration here are closely related to some recent investigations of Frobenius-Schur indicators for nonsemisimple Hopf algebras.  The 2nd Frobenius-Schur indicator $\nu_2(V)$ of an irreducible complex representation of a finite group was introduced in \cite{FS06}; the notion was then extended to semisimple Hopf algebras, quasi-Hopf algebras, certain $C^*$-fusion categories and conformal field theory (cf. \cite{LM00, MN05, FGSV99, Bantay97}).  Higher Frobenius-Schur indicators $\nu_n(V)$ for semisimple Hopf algebra have been extensively  studied in \cite{KSZ}.  In the most general context, Frobenius-Schur (FS-) indicators can be defined for each object $V$ in a \emph{pivotal} tensor category $\CC$, and they are invariants of these tensor categories \cite{NS1}.
\par

The $n$th Frobenius-Schur indicators $\nu_n(H)$ of the regular representation of a semisimple Hopf algebra $H$, defined in \cite{LM00}, in particular is an invariant of the fusion category $\C{H}$ (cf. \cite{NS1} and \cite[Thm. 2.2]{NS2}).  For this special representation it is obtained in \cite{KSZ} that
\begin{equation}\label{eq:nu_nH}
\nu_n(H)=\Tr(S\circ P_{n-1})
\end{equation}
where $P_k$ denotes the $k$th convolution power of the identity map $\id_H$ in $\End_{\k}(H)$. On elements, the map $S\circ P_{n-1}$ is given by $h\mapsto S(h_1\dots h_{n-1})$.

The importance of the FS-indicators is illustrated in their applications to semisimple Hopf algebras and \emph{spherical} fusion categories (see for examples \cite{BNRW}, \cite{DLN}, \cite{KSZ}, \cite{NS3}, \cite{NS4}, \cite{O15}   and \cite{Tuc16}). The arithmetic properties of the values of the FS-indicators have played an integral role in all these applications, and remains the main interest of FS-indicators (see for example  \cite{GM}, \cite{IMM}, \cite{MVW}, \cite{Schau16},  and \cite{Shim15}).

It would be tempting to extend the notion of FS-indicators for the study of finite tensor categories or nonsemisimple Hopf algebras.   One would expect that such a \emph{generalized} indicator for a general Hopf algebra $H$ should coincide with the existing one when $H$ is semisimple.
\par

The introduction of (what we refer to as) the KMN-indicators $\nuk_n(H)$ in \cite{KMN} is an attempt at this endeavor.  Note that the right hand side of \eqref{eq:nu_nH}, $\Tr(S\circ P_{n-1})$ is well-defined for any finite-dimensional Hopf algebra over any base field, and we denote it as $\nuk_n(H)$. It has been shown in \cite{KMN} that the scalar $\nuk_n(H)$ is an invariant of the finite tensor category $\C{H}$ for each positive integer $n$.  However, this definition of indicators for the regular representation in $H\mod$ cannot be extended to other objects in $\C{H}$.

In~\cite{shimizu15} Shimizu lays out an alternative categorial approach to generalized indicators for a nonsemisimple Hopf algebra $H$. He first constructs a {\it universal pivotalization} $(H\mod)^{\piv}$ of $H\mod$, i.e. a pivotal tensor category with a fixed monoidal functor $\Pi: (\C{H})^\piv \to \C{H}$ which is universal among all such categories. The pivotal category $(H\mod)^{\piv}$ has a \emph{regular object} $\RH$, and   the scalar $\nuk_n(H)$  can be recovered from a new version of the $n$th indicator $\nus_n(\RH^*)$. The  universal pivotalization is natural in the sense that for any monoidal functor $\FF: \C{H}\to \C{K}$, where $K$ is a Hopf algebra, there exists a unique pivotal functor $\FF^\piv :(\C{H})^\piv \to (\C{K})^\piv$ compatible with both $\Pi$ and $\FF$.

However, the invariance of $\nuk_n(H)$ does not follow immediately from this categorical framework. Instead, it would be a consequence of a proposed isomorphism $\FF^\piv(\RH) \cong \RK$ associated to any monoidal equivalence $\FF: \C{H}\to \C{K}$. While the latter condition remains open in general, we show below that the regular objects are preserved under monoidal equivalence for Hopf algebras with the Chevalley property.

\begin{theoremB}[Thm.~\ref{thm:presreg}]
Let $H$ and $K$ be Hopf algebras with the Chevalley property and $\FF:\C{H}\to \C{K}$ an equivalence of tensor categories.  Then the induced pivotal equivalence $\FF^{\piv}:(H\mod)^{\piv}\to (K\mod)^{\piv}$ on the universal pivotalizations satisfies $\FF^{\piv}(\mathbf{R}_H)\cong \mathbf{R}_K$.
\end{theoremB}

This gives a positive solution to~\cite[Quest. 5.12]{shimizu15}.  From Theorem {II} we recover the gauge invariance result of~\cite{KMN}, in the specific case of Hopf algebras with the Chevalley property.

\begin{corollaryB}[{cf.~\cite[Thm. 2.2]{KMN}}]
Suppose $H$ and $K$ are Hopf algebras with the Chevalley property and have equivalent tensor categories of representations.  Then $\nuk_n(H)=\nuk_n(K)$.
\end{corollaryB}

The paper is organized as follows: Section 2  recalls some basic notions and results on Hopf algebras and pivotal tensor categories. In Section 3, we prove that a specific element $\g_F$ associated to a Drinfeld twist $F$ of a semisimple Hopf algebra $H$ is fixed by the antipode of $H$, using the pseudounitary structure of $\C{H}$. We proceed to prove Theorem I and Corollary I in Section 4. In Section 5, we recall the construction of the universal pivotalization $(\C{H})^\piv$, the corresponding definition of $n$th indicators for an object in $(\C{H})^\piv$ and their relations to $\nuk_n(H)$. In Section 6, we introduce finite pivotalizations of $\C{H}$ and, in particular, the exponential pivotalization which contains all the possible pivotal categories defined on  $\C{H}$. In Section 7, we answer a question of Shimizu on the preservation of regular objects for Hopf algebras with the Chevalley property.

\section{Preliminaries} \label{s:gauge}

Throughout this paper,  we assume some basic definitions on Hopf algebras and monoidal categories. We denote the antipode of a Hopf algebra $H$ by $S_H$ or, when no confusion will arise, simply by $S$.  A tensor category in this paper is a $\k$-linear abelian monoidal category with simple unit object $\1$.  A  monoidal functor between two tensor categories is a pair $(\FF, \xi)$ in which $\FF$ is a $\k$-linear functor satisfying $\FF(\1)=\1$, and $\xi_{V, W} : \FF(V) \o \FF(W) \to \FF(V \o W)$ is the coherence isomorphism. If the context is clear, we may simply write $\FF$ for the pair $(\FF, \xi)$. The readers are referred to \cite{Kassel, Mont93bk} for the details.

\subsection{Gauge equivalence, twists, and the antipode}
\label{sect:ge}

Let $H$ be a finite-dimensional  Hopf algebra over $\k$ with antipode $S$, comultiplication $\Delta$ and counit $\e$.  The category $\C{H}$ of finite-dimensional representations of $H$ is a finite tensor categories in the sense of \cite{EO}. For $V \in \C{H}$, the dual vector space $V'$ of $V$ admits the natural right $H$-action $\leftharpoonup$ given by
$$
(v^*\leftharpoonup h)(v) = v^*(h v)
$$
for $h \in H$, $v^* \in V'$ and $v \in V$. The left dual $V^*$ of $V$ is the vector space $V'$ endowed with the left $H$-action defined by
$$
h v^*= v^* \leftharpoonup S(h)
$$
for $h \in H$ and $v^* \in V'$, with the usual evaluation $\ev: V^* \o V \to \k$ and the dual basis map as the coevaluation $\coev:\k \to  V \o V^*$. The right dual of $V$ is defined similarly with $S$ replaced by $S\inv$.

Suppose $K$ is another finite-dimensional Hopf algebra over $\k$ such that $\C{K}$ and $\C{H}$ are equivalent tensor categories. It follows from  \cite[Thm. 2.2]{NS2} that there is a gauge transformation $F = \sum_i f_i \o g_i \in H \o H$ (cf. \cite{Kassel}), which is an invertible element satisfying
$$
(\e \o \id) (F) = 1 = (\id \o \e)(F),
$$
such that the map $\Delta^F: H \to H \o H,  h\mapsto F\Delta(h)F\inv$ together with the counit $\e$ and the algebra structure of $H$ form a bialgebra $H^F$ and that $K \stackrel{\sigma}{\cong} H^F$ as bialgebras. In particular, $H^F$ is a Hopf algebra with the antipode give by
\begin{equation}\label{eq: antopode}
  S_F (h) = \b_F S(h) \b_F\inv
\end{equation}
where $\b_F = \sum_i f_i S(g_i)$. Following the terminology of \cite{Kassel} (cf. \cite{KMN}), we say that $K$ and $H$ are \emph{gauge equivalent} if the categories of their finite-dimensional representations are  equivalent tensor categories. A quantity $f(H)$ obtained from a finite-dimensional Hopf algebra $H$ is called a \emph{gauge invariant} if $f(H)=f(K)$ for any Hopf algebra $K$ gauge equivalent $H$. For instance, $\Tr(S)$ and $\Tr(S^2)$ are gauge invariants of $H$.

 If $F\inv = \sum_i d_i \o e_i$, then $\b_F\inv  = \sum_i S(d_i) e_i$. For the purpose of this paper, we set $\g_F = \b_F S(\b_F\inv)$ and so, by \eqref{eq: antopode}, we have
\begin{equation}\label{eq: antipode2}
  S_F^2 (h) = \g_F S^2(h) \g_F\inv
\end{equation}
for $h \in H$.

Since the associativities of $K$ and $H$ are given by $1\o 1\o1$, the gauge transformation $F$ satisfies the condition
\begin{equation}\label{eq:2-cocycle}
  (1 \o  F)(\id \o \Delta)(F) = (F \o  1)(\Delta \o \id)(F) \,.
\end{equation}
This is a necessary and sufficient condition for $\Delta^F$ to be coassociative. A gauge transformation $F \in H \o H$ satisfying equation \eqref{eq:2-cocycle} is often called a \emph{Drinfeld twist} or simply a \emph{twist}.

Suppose $F \in H \o H$ is a twist and $K \stackrel{\s}{\cong} H^F$ as Hopf algebras. Following \cite{Kassel}, one can define an equivalence $(\FF_\s, \xi^F): \C{H} \to \C{K}$ of tensor categories.  For $V \in \C{H}$,  $\FF_\s(V)$ is the left $K$-module with the action given by $k \cdot v:=\s(k)v$ for $k \in K$ and $v \in V$. The  assignment $V \mapsto \FF_\s(V)$ defines a $\k$-linear equivalence from $\C{H}$ to $\C{K}$ with identity action on the morphisms. Together with the natural isomorphism $\xi^F : \FF_\s(V) \o \FF_\s(W) \to \FF_\s(V \o W)$ defined by the action of $F\inv$ on $V \o W$, the pair $(\FF_\s, \xi^F): \C{H} \to \C{K}$ is an equivalence of tensor categories. If $K = H^F$ for some twist $F \in H\o H$, then
$(\Id, \xi^F): \C{H} \to \C{H^F}$ is an equivalence of tensor categories since $\FF_{\id}$ is the identity functor $\Id$.

\subsection{Pivotal categories}
\label{sect:piv}

For any finite tensor category $\CC$ with the unit object $\1$, the left duality can define a functor $(-)^*: \CC\to \CC^\op$ and the double dual functor $(-)^{**} : \CC \to \CC$ is an equivalence of tensor categories. A pivotal structure of $\CC$ is an isomorphism $j: \Id \to (-)^{**}$ of monoidal functors. Associated with a pivotal structure $j$ are the notions of \emph{trace} and \emph{dimension}: For any $V \in \CC$ and $f : V \to V$, one can define $\ptr(f)$ as the scalar of the composition
$$
\ptr(f):=(\1 \xrightarrow{\coev} V \o V^* \xrightarrow{f \o V^*} V\o V^*\xrightarrow{j\ox V^*} V^{**} \o V^* \xrightarrow{\ev} \1)
$$
and $d(V) =\ptr(\id_V)$. A finite tensor category with a specified pivotal structure is called a\emph{ pivotal category}.

Suppose $\CC$ and $\DD$ are pivotal categories with the pivotal structures $j$ and $j'$ respectively, and $(\FF, \xi): \CC \to \DD$ is a monoidal functor. Then there exists a unique natural isomorphism $\tilde\xi: \FF(V^*) \to \FF(V)^*$ which is determined by either of the following commutative diagrams (cf. \cite[p67]{NS1}):
 \begin{equation} \label{eq:dual_tran}
 {\scriptsize
\begin{gathered}\xymatrix{  \FF(V^*) \o \FF(V) \ar[d]^-{\xi}  \ar[r]^-{\tilde\xi \o \FF(V)} & \FF(V)^* \o \FF(V)  \ar[d]_-{\ev} \\
 \FF(V^* \o V) \ar[r]^-{\FF(\ev)}  & \1 \,
}
\end{gathered}
\text{\normalsize or }
\begin{gathered}\xymatrix{  \FF(V) \o \FF(V^*)  \ar[r]^-{\FF(V)\o \tilde\xi} & \FF(V) \o \FF(V)^*   \\
 \FF(V \o V^*) \ar[u]_-{\xi^{-1}}   & \1\ar[l]^-{\FF(\coev)} \ar[u]_-{\coev} \,.
}
\end{gathered}
}
 \end{equation}
The monoidal functor $(\FF, \xi)$ is said to be \emph{pivotal} if it preserves the pivotal structures, which means the commutative diagram
\begin{equation}\label{eq:piv_eqv}
\begin{gathered}
  \xymatrix{
\FF(V) \ar[d]_-{j'_{\FF(V)}} \ar[r]^-{\FF(j_V)} & \FF(V^{**}) \ar[d]^-{\tilde\xi} \\
\FF(V)^{**} \ar[r]^-{{\tilde\xi}^*}& \FF(V^*)^*
}
\end{gathered}
\end{equation}
is satisfied for $V \in \CC$. It follows from \cite[Lem. 6.1]{NS1} that pivotal monoidal equivalence preserves dimensions. More precisely, if $\FF: \CC \to \DD$ is an equivalence of pivotal  categories, then $d(V)= d(\FF(V))$ for $V  \in \CC$.

 \section{Semisimple Hopf algebras and pseudounitary fusion categories}
\label{sect:pseudounitary}

In general, a finite tensor category may not have a pivotal structure. However, all the known semisimple finite tensor categories, also called \emph{fusion categories}, over $\k$, admits a pivotal structure. It remains to be an open question whether every fusion category admits a pivotal structure (cf. \cite{ENO}).  We present an equivalent definition of \emph{pseudounitary} fusion categories obtained in \cite{ENO} or more generally in \cite{DGNO} as in the following proposition.

\begin{proposition} \label{p:0} \cite{ENO} Let $\k_c$ denote the subfield of $\k$ generated by $\BQ$ and all the roots of unity in $\k$. A fusion category $\CC$ over $\k$ is called ($\phi$-)pseudounitary if there exist a pivotal structure $j^\CC$ and a field monomorphism $\phi: \k_c \to \BC$ such that $\phi(d(V))$ is real and nonnegative for all simple $V \in \CC$, where $d(V)$ is the dimension of $V$ associated with $j^\CC$. In this case, this pivotal structure $j^\CC$ is unique and $\phi(d(V))$ is identical to the Frobenius-Perron dimension of $V$.
\end{proposition}

The reference of $\phi$ becomes irrelevant when the dimensions associated with the pivotal structure $j^\CC$ of $\CC$ are nonnegative integers.  In this case, $\CC$ is simply said to be pseudounitary, and $j^\CC$ is called the \emph{canonical} pivotal structure of $\CC$. In particular, the fusion category $\C{H}$ of a finite-dimensional semisimple quasi-Hopf algebra $H$ is pseudounitary and the pivotal dimension of an $H$-module $V$ associated with the canonical pivotal structure of $\C{H}$ is simply the ordinary dimension of $V$ (cf. \cite{ENO}).
\par
The canonical pivotal structure $\jj$ on the \emph{trivial} fusion category  $\Vec$ of finite-dimensional $\k$-linear space is just the usual vector space isomorphism $V\to V^{**}$, which sends an element $v\in V$ to the evaluation function $\hat{v}: V^*\to \k$, $f \mapsto f(v)$.

 Let $H$ be a finite-dimensional semisimple Hopf algebra over $\k$. Then the antipode $S$ of $H$ satisfies $S^2=\id$ (cf. \cite{LR88}). Thus, for $V \in \C{H}$, the natural isomorphism $\jj: V \to V^{**}$ of vector space is an $H$-module map. In fact, $\jj$ provides a pivotal structure of $\C{H}$ and the associated pivotal dimension $d(V)$ of $V$, given by the composition map
 $$
 \k \xrightarrow{\mathrm{coev}} V \o V^* \xrightarrow{j \o V^*} V^{**}\o V^* \xrightarrow{\ev} \k,
 $$
 is equal to its ordinary dimension $\dim V$, which is a nonnegative integer. Therefore, $\jj$ is the canonical pivotal structure of $\C{H}$.

  By \cite[Cor. 6.2]{NS1}, the canonical pivotal structure of a pseudounitary fusion category is preserved by any monoidal equivalence of fusion categories.  For the purpose of this article, we restate this statement in the context of semisimple Hopf algebras.

\begin{corollary}\cite[Cor. 6.2]{NS1}\label{c:1}
Let $H$ and $K$ be finite-dimensional semisimple Hopf algebras over $\k$. If $(\FF, \xi): \C{H}\to\C{K}$ defines a monoidal equivalence, then $(\FF, \xi)$ preserves their canonical pivotal structures, i.e. they satisfy the commutative diagram \eqref{eq:piv_eqv}. In particular, if $K \stackrel{\sigma}{\cong} H^F$  as Hopf algebras for some twist $F \in H \o H$, then the monoidal equivalence $(\FF_\s, \xi^F): \C{H} \to \C{K}$ preserves their canonical pivotal structures.
\end{corollary}

Now, we can prove the following condition on a twist of a semisimple Hopf algebra.
\begin{theorem} \label{t:1}
  Let $H$ be a semisimple Hopf algebra over $\k$ with antipode $S$, $F=\sum_i f_i \o g_i \in H \o H$ a twist and $\b_F=\sum_i f_i S(g_i)$.
  Then $$S(\b_F) = \b_F.$$
\end{theorem}

\begin{proof}
  Let $F\inv = \sum_i d_i \o e_i$.  Then $\b\inv = \sum_i S(d_i) e_i$  (cf. Section \ref{s:gauge}), where $\b_F$ is simply abbreviated as $\b$. For  $V \in \C{H}$, we denote by $V^*$ and $V\du$ respectively the left duals of $V$ in $\C{H}$ and $\C{H^F}$. It follows from \eqref{eq:dual_tran} that the duality transformation $\tilde\xi^F: V^* \to V\du$, for $V \in \C{H}$, of the monoidal equivalence $(\Id, \xi^F): \C{H} \to \C{H^F}$, is given by
  \begin{equation} \label{eq:dual_tran2}
  \tilde\xi^F(v^*)= v^*\leftharpoonup\b\inv
  \end{equation}
   for all $v^* \in V^*$.
Since both $H$ and $H^F$ are semisimple, their canonical pivotal structures are the same as the usual natural isomorphism $\jj$ of finite-dimensional vector spaces over $\k$. Since $(\Id, \xi^F)$ preserves the canonical pivotal structures, by \eqref{eq:piv_eqv}, we have
 $$
 \begin{aligned}
  \tilde\xi^F (\jj(v))(v^*) & =  (\tilde\xi^F)^* (\jj(v))(v^*)\\ & = \jj(v)(\tilde{\xi}^F(v^*))
  =(v^*\leftharpoonup \b\inv)(v)= v^*(\b\inv v)
 \end{aligned}
 $$
 for all $v \in V$ and $v^* \in V^*$. Rewriting the first term of this equation, we find
 $$
 v^*(S(\b\inv)v)  = v^*(\b\inv v)\,.
 $$
 This implies $\b\inv = S(\b\inv)$ by taking $V=H$ and $v = 1$.
\end{proof}

\section{Hopf algebras with the Chevalley property}
A finite-dimensional Hopf algebra $H$ over $\k$ is said to have the \emph{Chevalley property} if the Jacobson radical $J(H)$ of $H$ is a Hopf ideal. In this case, $\ol H =H/J(H)$ is a semisimple Hopf algebra and the natural surjection $\pi: H\to \ol H$ is a Hopf algebra map.  Let $F \in H\o H$ be a twist of $H$. Then
$$
\ol F := (\pi \o \pi)(F) \in \ol H \o \ol H
$$
is a twist and so
$$
\pi(\b_F) = \b_{\ol F}= \ol{S}(\b_{\ol F}) = \pi(S(\b_F))
$$
by Theorem \ref{t:1}, where $\ol S$ denotes the antipode of $\ol H$. Therefore, $S(\b_F) \in \b_F +J(H)$ and this proves
\begin{lemma}\label{l:reduction}
  Let $H$ be a finite-dimensional Hopf algebra over $\k$ with the Chevalley property. For any twist $F\in H\o H$,
  $$
  S(\b_F) \in \b_F + J(H). \qedhere
  $$
\end{lemma}
We will need the following lemma.
\begin{lemma} \label{l:trace}
 Let $A$ be a finite-dimensional algebra over $\k$ and $T$ an algebra endomorphism or anti-endomorphism of $A$.
\begin{enumerate}
  \item[(i)] For any $x \in J(A)$ and $a \in A$,
  $$
  l(x) r(a) T, \quad  l(a) r(x) T
  $$
  are nilpotent operators, where $l(x)$ and $r(x)$ respectively denote the left and the right multiplication by $x$.
  \item[(ii)]  For any $a,a', b, b' \in A$ such that $a' \in a+ J(A)$ and $b' \in b + J(A)$,  we have
  $$
  \Tr(l(a)  r(b)  T) = \Tr(l(a')  r(b')   T)\,.
  $$
\end{enumerate}
\end{lemma}
\begin{proof}
\noindent (i)
  Let $n$ be a positive integer such that $J(A)^n=0$. We first consider the case when $T$ is an algebra endomorphism of $A$. Then
  \begin{eqnarray*}
   (l(a)   r(x)   T)^n & = & l(a)  l(T(a))   \cdots   l(T^{n-1}(a))    r(x)  \cdots   r(T^{n-1}(x) )    T^n \\
   & = & l(a T(a)\cdots T^{n-1}(a))    r(T^{n-1}(x)  \cdots   T(x)    x)  T^n\,.
  \end{eqnarray*}
  Since $J(A)^n=0$ and $x, T(x), \dots, T^{n-1}(x) \in J(A)$,
  $$
  T^{n-1}(x) \cdots T(x) x  =0.
  $$
   Therefore,  $(l(a)   r(x)   T)^n =0$. By the same argument, one can show that
  $(l(x)   r(a)   T)^n =0$. In particular, they are nilpotent operators.

  If $T$ is an algebra anti-endomorphism of $A$, then
  $$(l(a)r(x)T)^2= l(aT(x)) r(T(a) x) T^2.
   $$
   Since $T^2$ is an algebra endomorphism of $A$ and $aT(x) \in J(A)$, we have $(l(a)r(x)T)^{2n}=0$. Similarly, $(l(x)r(a)T)^{2n}=0$.\\

\noindent (ii) Let $a'=a+x$ and $b'=b+y$ for some $x, y\in J(A)$.
$$
l(a') r(b') T = l(a) r(b) T + l(x) r(b') T + l(a)r(y) T\,.
$$
By (i), $l(x) r(b') T$ and  $l(a)r(y) T$ are nilpotent operators, and so the result follows.
\end{proof}

We can now prove that the traces of the powers of the antipode of a Hopf algebra with the Chevalley property are gauge invariants.
\begin{theorem}\label{t:2}
  Let $H$ be a Hopf algebra over $\k$ with the antipode $S$. Suppose $H$ has the Chevalley property. Then for any twist $F \in H \o H$, we have
  $$
  \Tr(S^n_F) = \Tr(S^n)
  $$
  for all integers $n$, where $S_F$ is the antipode of $H^F$. Moreover, if $K$ is another Hopf algebra over $\k$ with antipode $S'$ which is gauge equivalent to $H$, then
  $$
  \Tr(S^n)=\Tr({S'}^n)
  $$
  for all integers $n$.
\end{theorem}
\begin{proof}
   By \eqref{eq: antopode},
  the antipode $S_F$  of $H^F$ is given by
  $$
  S_F(h) = \b_F S(h) \b_F\inv
  $$
  for $h \in H$. Recall from \eqref{eq: antipode2} that
   $$
  S^2_F(h) = \g_F S^2(h) \g_F\inv
  $$
  where $\g_F = \b_F S(\b_F\inv)$. Then, for any nonnegative integer $n$, we can write $S_F^n = l(u_n) r(u_n\inv) S^n$ where $u_0=1$ and
  $$
  u_n =  \left\{\begin{array}{ll}
\g_F S^2(\g_F) \cdots S^{n-2}(\g_F) & \mbox{ if $n$ is positive and even}, \\
\b_F S(u_{n-1}\inv) & \mbox{ if $n$ is odd}.
 \end{array}\right.
 $$
 Thus, if $n$ is an even positive integer, $u_n  \in 1+ J(H)$ by Lemma \ref{l:reduction}. It follows from Lemma \ref{l:trace} that
 $$
 \Tr(S_F^n) = \Tr(l(u_n) r(u_n\inv) S^n) = \Tr(l(1) r(1) S^n) = \Tr(S^n)\,.
 $$

 From now, we assume $n$ is odd. Then $u_n \in \b_F + J(H)$ and so we have
 \begin{align} \label{eq:step1}
   \Tr(S_F^n) &= \Tr(l(u_n) r(u_n\inv) S^n)= \Tr(l(\b_F) r(\b_F\inv) S^n )\nonumber\\
    & = \Tr(l(\b_F) r(S^n(\b_F\inv)) S^n )\,.
 \end{align}
 The last equality of the above equation follows from Lemmas \ref{l:reduction} and \ref{l:trace}(ii).

 Let $\L$ be a left integral of $H$ and $\l$ a right integral of $H^*$ such that $\l(\L)=1$.  By \cite[Thm. 2]{Radf94},
 $$
 \Tr(T) = \l(S(\L_2) T(\L_1))
 $$
 for any $\k$-linear endomorphism $T$ on $H$,  where $\Delta(\L) = \L_1 \o \L_2$ is the Sweedler notation with the summation suppressed.  Thus, by \eqref{eq:step1}, we have
  \begin{equation} \label{eq:step2}
   \Tr(S_F^n)  = \l(S(\L_2) \b_F S^n(\L_1) S^n(\b_F\inv)) =  \l(S(\L_2) \b_F S^n(\b_F\inv \L_1))  \,.
  \end{equation}
  Recall from \cite[p591]{Radf94} that
  $$
  \L_1 \o a \L_2 = S(a) \L_1 \o \L_2
  $$
  for all $a \in H$. Using this equality and \eqref{eq:step2}, we find
  \begin{align*}
    \Tr(S^n_F)& =\l(S(\L_2) \b_F S^n(\b_F\inv \L_1))  = \l(S(S\inv (\b_F\inv)\L_2) \b_F S^n(\L_1)) \\
    & = \l(S(\L_2) \b_F\inv \b_F S^n(\L_1)) =\l(S(\L_2) S^n(\L_1))  =\Tr(S^n)\,.
  \end{align*}
The second statement of the theorem then follows immediately from Corollary \ref{c:1}.
\end{proof}
\begin{corollary}\label{cor:ordpres}
  If $H$ is a finite-dimensional Hopf algebra over $\k$ with the Chevalley property, then $\ord(S)$ is a gauge invariant. In particular, $\ord(S^2)$ is a gauge invariant.
\end{corollary}
\begin{proof}
Since $\k$ is of characteristic zero, $\Tr(S^n) = \dim H$ if and only if $S^n=\id$. In particular, $\ord(S)$ is the smallest positive integer $n$ such that $\Tr(S^n)=\dim H$. If $K$ is a Hopf algebra (over $\k$) with the antipode $S'$ and is gauge equivalent to $H$, then $\dim K = \dim H$ by Corollary \ref{c:1}. Hence, by Theorem \ref{t:2}, $\ord(S) = \ord(S')$. Note that $S$ has odd order if, and only if, $S$ is the identity. Therefore, the last statement follows.
\end{proof}

\section{Pivotalization and indicators}
\label{sect:indicators}

\subsection{KMN-indicators}

For the regular representation $H$ of a semisimple Hopf algebra $H$ over $\k$ with the antipode $S$, the formula of the $n$th Frobenius-Schur indicator $\nu_n(H)$  was obtained in \cite{KSZ} and is given by \eqref{eq:nu_nH}. Since a monoidal equivalence between the module categories of two finite-dimensional Hopf algebras preserves their regular representation \cite[Thm. 2.2]{NS2} and Frobenius-Schur indicators are invariant under monoidal equivalences (cf. \cite[Cor. 4.4]{NS1} or \cite[Prop. 3.2]{NS2}), $\nu_n(H)$ is an invariant of $\Rep(H)$ if $H$ is semisimple.

The formula \eqref{eq:nu_nH} is well-defined even for a nonsemisimple Hopf algebra $H$ without any pivotal structure in $\C{H}$. In fact, the gauge invariance of these scalars has been recently proved in \cite{KMN} which is stated as the following theorem.

\begin{theorem}[{\cite[Thm. 2.2]{KMN}}]\label{thm:KMN}
For any finite-dimensional Hopf algebra $H$ over any field $\k$, we define $\nu_n^{\mathrm{KMN}}(H)$ as in \eqref{eq:nu_nH}.  If $H$ and $K$ are gauge equivalent finite-dimensional Hopf algebras over $\k$, then we have
$$\nuk_n(H)=\nuk_n(K).$$
\end{theorem}

In general, these indicators $\nuk_n(H)$ can {\it only} be defined for the regular representation of $H$. The proof Theorem \ref{thm:KMN} relies heavily on Corollary \ref{c:1} and theory of Hopf algebras. We would like to have a categorial framework for the definition of $\nuk_n(H)$ in order to extend the definitions of the indicators to other objects in $\C{H}$ and give a categorial proof of gauge invariance of these indicators.

\subsection{The universal pivotalization}

In \cite{shimizu15} the notion of universal pivotalization $\CCp$ of a finite tensor category $\CC$ is proposed in order to produce indicators for pairs consisting of an object $V$ in $\mathscr{C}$ along with a chosen isomorphism to its double dual. Under this categorical framework, $\nuk_n(H)$ is the $n$th indicator of a special (or regular) object in $(\C{H})^\piv$.  We recall some constructions and results from~\cite{shimizu15} here.
\par

For a finite tensor category $\CC$ one can construct the universal pivotalization $\Pi_\CC: \CCp\to \CC$ of $\CC$, which is referred to as the {\it pivotal cover} of $\CC$ in~\cite{shimizu15}.\footnote{We accept the term pivotal cover, but adopt the term pivotalization as it is consistent with the constructions of~\cite{EGNObook} and admits adjectives more readily.}  The category $\CCp$ is the abelian, rigid, monoidal category of pairs $(V,\phi_V)$ of an object $V$ and an isomorphism $\phi_V:V\to V^{**}$ in $\CC$.  Morphisms $(V,\phi_V)\to (W,\phi_W)$ in $\CCp$ are maps $f:V\to W$ in $\CC$ which satisfy $\phi_W f=f^{**}\phi_V$.  Note that the forgetful functor $\Pi_\CC: \CCp\to \CC$ is faithful.
\par

The category $\CCp$ will be monoidal under the obvious tensor product $(V,\phi_V)\ox(W,\phi_W):=(V\ox W,\phi_V\ox \phi_W)$ (where we suppress the natural isomorphism $(V\ox W)^{**}\cong V^{**}\ox W^{**}$), and (left) rigid under the dual $(V,\phi_V)^*=(V^*,(\phi_V^{-1})^*)$.  There is a natural pivotal structure $j:\Id_{\CCp}\to (-)^{**}$ on $\mathscr{C}^{\piv}$ which, on each object $(V,\phi_V)$, is simply given by $j_{(V,\phi_V)}:=\phi_V$.
\par

The construction $\CCp$ is universal in the sense that any monoidal functor $\FF:\DD\to \CC$ from a pivotal tensor category $\DD$ factors uniquely through $\CCp$.  By faithfulness of the forgetful functor $\Pi_\CC: \CCp\to \CC$, the factorization $\tilde{\FF}: \DD\to \CCp$, which is a monoidal functor preserving the pivotal structures, is determined uniquely by where it sends objects.  This factorization is described as  follows.

\begin{theorem}[{\cite[Thm. 4.3]{shimizu15}}] \label{thm:piv_form}
Let $j$ denote the pivotal structure on $\DD$ and $(\FF, \xi): \DD \to \CC$ a monoidal functor.  Then the factorization $\tilde{\FF}: \DD\to \CCp$ sends each object $V$ in $\DD$ to the pair $(\FF(V), (\tilde{\xi}^*)^{-1}\tilde{\xi}\FF(j_V))$, where $\tilde{\xi}$ is the duality transformation as in Section~\ref{s:gauge}.
\end{theorem}

From the universal property for $\CCp$ one can conclude that the construction $(-)^{\piv}$ is functorial, which means a monoidal functor $\FF:\DD\to\CC$ induces a unique pivotal functor $\FF^{\piv}:\DDp\to \CCp$ which satisfies the commutative diagram
$$
\xymatrix{\DDp \ar[r]^-{\FF^{\piv}} \ar[d]_-{\Pi_\DD} & \CCp \ar[d]^-{\Pi_\CC}\\
\DD \ar[r]^-{\FF}  & \CC
}
$$
of monoidal functors.

\subsection{Indicators via $\mathscr{C}^{\piv}$}
\label{sect:nuSh}

Following \cite{NS1}, for any $V, W \in \CC$, we denote by $A_{V, W}$ and $D_{V, W}$ for the natural isomorphisms  $\Hom_\CC(\1, V \o W) \to \Hom_\CC(V^*, W)$ and $\Hom_\CC(V, W) \to \Hom_\CC(W^*, V^*)$ respectively. Thus,
$$
T_{V,W} :=  A_{W, V^{**}}\inv \circ D_{V^*, W} \circ A_{V, W}
$$
is a natural isomorphism from $\Hom_\CC(\1, V \o W) \to \Hom_\CC(\1, W \o V^{**})$.  We also define $V^{\o 0} = \1$ and $V^{\o n} = V \o V^{\o (n-1)}$ for any positive integer $n$ inductively.

Similar to the definition provided in \cite[p71]{NS1}, for any $\bV = (V, \phi_V) \in \CCp$ and positive integer $n$, one can define
the map $E_\bV^{(n)}: \Hom_\CC(\1, V^{\o n}) \to   \Hom_\CC(\1, V^{\o n})$ by
$$
E_\bV^{(n)}(f):= \Phi^{(n)}\circ (\id \o \phi_V\inv) \circ  T_{V, W}(f)
$$
where $W = V^{\o (n-1)}$ and $\Phi^{(n)} : W \o V \to V \o W$ is the unique map obtained by the associativity isomorphisms. Shimizu's version of the $n$th FS-indicator of $\bV$ is defined as
$$
\nus_n(\bV) = \Tr(E_\bV^{(n)})\,.
$$
This indicator is preserved by monoidal equivalence in the following sense:

\begin{theorem}[{\cite[Thm. 5.3]{shimizu15}}]
If $\FF : \CC \to \DD$ is an equivalence of monoidal categories, for any $\bV \in \CCp$ and positive integer $n$, we have $\nus_n(\bV) = \nus_n(\FF^\piv(\bV))$.
\end{theorem}

\begin{remark}
The definition of the $n$th FS-indicator $\nus_n(\bV)$ of $\bV$ is different from the definition $\nu_n(\bV)$ introduced in \cite{NS1}, in which $E_\bV^{(n)}$ is defined on the space $\Hom_{\CCp}(\1, \bV^{\o n})$ instead. It is natural to ask the question whether or how these two notions of indicators are related.
\end{remark}

In the case of a finite-dimensional Hopf algebra $\CC=H\mod$, we take $\mathbf{R}_H=(H, \vp_H)$ to be the object in $\CCp$, in  which $H$ is the left regular  $H$-module and $ \vp_H: H\to H^{**}$ is the composition $\jj\circ S^2:H\to \FF_{S^2}(H)\cong H^{**}$.  We call $\mathbf{R}_H$ the {\it regular object} in $\CCp$, and  we have the following

\begin{theorem}[{\cite[Thm. 5.7]{shimizu15}}]
Suppose $\CC=H\mod$.  Then for each integer $n$ we have $\nus_n(\mathbf{R}^*_H)=\nu_n^{\mathrm{KMN}}(H)$.
\end{theorem}

The theorem provides a convincing argument to pursue this categorical framework of FS-indicator for nonsemisimple Hopf algebras. However, this framework does not yield another proof for the gauge invariance of $\nu_n^{\mathrm{KMN}}(H)$ (cf. Theorem \ref{thm:KMN}). The gauge invariance of $\nu_n^{\mathrm{KMN}}(H)$ will follow if the following question raised in~\cite{shimizu15} can be positively answered.

\begin{question}[\cite{shimizu15}]\label{quest:shimizu}
Let $H$ and $K$ be two gauge equivalent Hopf algebras, and let
 $\FF:\C{H} \to \C{K}$ be a monoidal equivalence.  Do we have $\FF^{\upiv}(\mathbf{R}_H)\cong \mathbf{R}_K$ in $(\C{K})^\upiv$?
\end{question}

If the question is affirmatively answered for gauge equivalent Hopf algebras $H$ and $K$, then we have
$\FF^\piv(\RH) \cong \RK$ in $(\C{K})^\piv$ for any monoidal equivalence $\FF: \C{H} \to \C{K}$. Thus, $$\FF^\piv(\RH^*) \cong (\FF^\piv(\RH))^*  \cong \RK^*.$$ 
It follows from \cite[Thm. 5.3]{shimizu15} that
$$
\nuk_n(H) = \nus_n(\RH^*) = \nus_n(\FF^\piv(\RH^*)) = \nus_n(\RK^*) = \nuk_n(K)\,.
$$

An affirmative answer to the question for semisimple $H$ has been provided in \cite[Prop 5.10]{shimizu15}, and we will give in Theorem~\ref{thm:presreg} a positive answer for $H$ having the Chevalley property.  As discussed above, an affirmative answer to the above question yields a categorial proof of Theorem~\ref{thm:KMN}.

\section{Finite pivotalizations for Hopf algebras}
\label{sect:Npiv}

Let $\CC = H\mod$. In this section we simply remark that the universal pivotalization $\mathscr{C}^{\piv}$, which is not a finite tensor category in general, has a finite alternative for module categories of Hopf algebras.
\par

For any $\k$-linear map $\tau:V\to V^{**}$ we let $\underline{\tau}\in\Aut_\k(V)$ denote the automorphism $\underline{\tau} := (\jj)^{-1} \circ \tau$.

\begin{definition}
For a Hopf algebra $H$ we let $H^{\piv}$ denote the smash product $H\rtimes \mathbb{Z}$, where the generator $x$ of $\mathbb{Z}$ acts on $H$ by $S^2$.  Similarly, for any positive integer $N$ with $\mathrm{ord}(S^2)|N$, we take $H^{\Npiv}=H\rtimes (\mathbb{Z}/N\mathbb{Z})$, where again the generator $x$ of $\mathbb{Z}/N\mathbb{Z}$ acts as $S^2$.
\end{definition}

The smash products $H^{\piv}$ and $H^{\Npiv}$ admit a unique Hopf structure so that the inclusions $H\to H^{\piv}$ and $H\to H^{\Npiv}$ are Hopf algebra maps and $x$ is grouplike.
\par

It has been pointed out in \cite[Rem. 4.5]{shimizu15} that $\C{H^\piv}$ is isomorphic to $(\C{H})^\piv$ as pivotal tensor categories. To realize the identification $\Theta: \C{H^\piv} \xrightarrow{\cong} \CCp$ one takes an $H^\piv$-module $V$ to the $H$-module $V$ along with the isomorphism $\phi_V:=\jj\circ l(x):V\to \FF_{S^2}(V)\cong V^{**}$. On elements, $\phi_V(v)=\jj(x\cdot v)$.  So we see that the inverse functor $\Theta\inv:\CCp\to H^\piv\mod$ takes the pair $(V,\phi_V)$ to the $H$-module $V$ along with the action of the grouplike $x\in H^\piv$ by $x\cdot v=\underline{\phi_V}(v)$.
\par

From the above description of $\CCp$ for Hopf algebras we see that $\CCp$ will not usually be a finite tensor category.
\par

Note that, for any integer $N$ as above, we have the Hopf projection $H^\upiv\to H^{N\piv}$ which is the identity on $H$ and sends $x$ (in $H^\upiv$) to $x$ (in $H^{N\piv}$).  Dually, we get a fully faithful embedding of tensor categories $\C{H^{N\piv}}\to H^\upiv\mod$.

\begin{definition}
For any positive integer $N$ which is divisible by the order of $S^2$, we let $\CCNp$ denote the full subcategory of $\CCp$ which is the image of $\C{H^{N\piv}}\subset H^\upiv\mod$ along the isomorphism $\Theta:H^\upiv\mod\to\CCp$.
\end{definition}

From this point on if we write $H^{N\piv}$ or $\CCNp$ we are assuming that $N$ is a positive integer with $\mathrm{ord}(S^2)|N$.  We see, from the descriptions of the isomorphisms $\Theta$ and $\Theta^{-1}$ given above, that $\mathscr{C}^{\Npiv}$ is the full subcategory consisting of all pairs $(V,\phi_V)$ so that the associated automorphism $\underline{\phi_V}\in\Aut_\k(V)$ has order dividing $N$.

\begin{lemma}
The category $\CCNp$ is a pivotal finite tensor subcategory in the pivotal (non-finite) tensor category $\CCp$ which contains $\mathbf{R}_H$.
\end{lemma}

\begin{proof}
Since the map $\Theta:H^\upiv\mod\to \CCp$ is a tensor equivalence, it follows that $\CCNp$, which is defined as the image of $\C{H^{N\piv}}$ in $\CCp$, is a full tensor subcategory in $\CCp$.  The category $\CCNp$ is pivotal with its pivotal structure  inherited from $\CCp$.  The fact that $\mathbf{R}_H=(H,\jj\circ S^2)$ is in $\mathscr{C}^{\Npiv}$ just follows from the fact the order of $S^2=\underline{\phi_{\mathbf{R}_H}}$ is assumed to divide $N$.
\end{proof}

\begin{remark}
There is another interesting object $\mathbf{A}_H$ introduced in~\cite[Sect. 6.1 \& Thm. 7.1]{shimizu15}.  This object is the adjoint representation $H_{\mathrm{ad}}$ of $H$ along with the isomorphism $\phi_{\mathbf{A}_H}=\jj\circ S^2$.  We will have that $\mathbf{A}_H$ is also in $\mathscr{C}^{\Npiv}$ for any $N$.
\end{remark}

Some choices for $N$, which are of particular interest, are $N=\mathrm{ord}(S^2)$ or $N=\qexp(H)$, where $\qexp(H)$ is the quasi-exponent of $H$.  Recall that the quasi-exponent $\qexp(H)$ of $H$ is defined as the unipotency index of the Drinfeld element $u$ in the Drinfeld double $D(H)$ of $H$ (see~\cite{EGqexp02}).  This number is always finite and divisible by the order of $S^2$~\cite[Prop 2.5]{EGqexp02}. More importantly, $\qexp(H)$ is a gauge invariant of $H$.
\par

When we would like to pivotalize with respect to the the quasi-exponent we take $H^{\Epiv}=H^{\qexp(H)\mathrm{piv}}$ and $\CC^{\Epiv}=\CC^{\qexp(H)\mathrm{piv}}$.  We call $\mathscr{C}^{\Epiv}$ the {\it exponential pivotalization} of $\CC=H\mod$.
\par

If $\CC$ admits any pivotal structures, one can show that the exponential pivotalization contains a copy of $(\mathscr{C},j)$ for any choice of pivotal structure $j$ on $\mathscr{C}$ as a full pivotal subcategory.  More specifically, for any choice of pivotal structure $j$ on $\mathscr{C}$ the induced map $(\mathscr{C},j)\to \mathscr{C}^{\piv}$ will necessarily have image in $\mathscr{C}^{\Epiv}$.  In this way, the indicators for $\mathscr{C}$ calculated with respect to any choice of pivotal structure can be recovered from the (Shimizu-)indicators on $\mathscr{C}^{\Epiv}$.
\par

For some Hopf algebras $H$, the integer $\qexp(H)$ is minimal so that $\mathscr{C}^{\Npiv}$ has this property.  For example, when we take the generalized Taft algebra
\[
H_{n,d}(\zeta)=k\langle g,x\rangle/(g^{nd}-1,x^d,gx-\zeta xg),
\]
where $\zeta$ is a {\it primitive} $d$th root of unity (cf. \cite{Taft71},\cite[Def 3.1]{etingofwalton16}).  We have $\mathrm{ord}(S^2)=d$ and $nd=\qexp(H_{n,d}(\zeta))$ by \cite[Thm. 4.6]{EGqexp02}.  The grouplike element $g$ provides a pivotal structure $j$ on $\C{H_{n,d}(\zeta)}$, and the resulting map into $(\C{H_{n,d}(\zeta)})^{\piv}$ has image in $(\C{H_{n,d}(\zeta)})^{\Npiv}$ if and only if $\qexp(H_{n,d}(\zeta))|N$.  This relationship can be seen as a consequence of the general fact that $\qexp(H)=\qexp(G(H))$ for any pointed Hopf algebra $H$~\cite[Thm. 4.6]{EGqexp02}.
\par

Our functoriality result for the finite pivotalizations is the following.

\begin{proposition}\label{prop:Npiv}
For any monoidal equivalence $\FF:H\mod\to K\mod$, where $H$ and $K$ are Hopf algebras, the functor $\FF^{\piv}$ restricts to an equivalence $\FF^{\Epiv}:(H\mod)^{\Epiv}\to (K\mod)^{\Epiv}$.  Furthermore, when $H$ has the Chevalley property $\FF^{\piv}$ restricts to an equivalence $\FF^{\Npiv}:(H\mod)^{\Npiv}\to (K\mod)^{\Npiv}$ for each $N$ (in particular $N=\mathrm{ord}(S^2_H)=\mathrm{ord}(S^2_K)$).
\end{proposition}

The proof of the proposition is given in the appendix.

\section{Preservation of the regular object}

In this section we show that for a monoidal equivalence $\FF:H\mod\to K\mod$ of Hopf algebras $H$ and $K$ with the Chevalley property we will have $\FF^{\piv}(\mathbf{R}_H)\cong \mathbf{R}_K$.  From this we recover Theorem \ref{thm:KMN} for Hopf algebras with the Chevalley property.
\par

Let $H$ be a finite-dimensional Hopf algebra with antipode $S$, and  $F\in H\o H$ a twist of $H$.
We let $\CC=\C{H}$, $\CC_F=\C{H^F}$, and let $F=(\FF_{\id}, \xi^F)$ denote the associated equivalence from $\CC$ to  $\CC_F$, by abuse of notation.
\par

For this section we will be making copious use of the isomorphism $\jj:V\to V^{\ast\ast}$, and adopt the shorthand $\hat{v}=\jj(v)\in V^{\ast\ast}$ for $v\in V$.  Recall that $\hat{v}$ is just the evaluation map $V^\ast\to \k$, $\eta\mapsto \eta(v)$.

\subsection{Preservation of regular objects}

Recall that the antipode $S_F$ of $H^F$ is given by $S_F(h) = \b_F S(h)\b_F\inv$ and that $\g_F=\beta_FS(\beta_F)^{-1}$. For any positive integer $k$, define
$$
\g_F^{(k)} = \g_F S^2(\g_F)\cdots  S^{2k-2}(\g_F)\,.
$$
Then we have $S^{2k}_F(h) = \g_F^{(k)} S^{2k}(h) (\g_F^{(k)})\inv$ for all positive integers $k$ and $h\in H$.  The following lemma is well-known and it follows immediately from \cite[Eq.(6)]{AEGN02}.

\begin{lemma}
The element $\g_F^{(\mathrm{ord}(S^2))}$ is a grouplike element in $H^F$.
\end{lemma}

\begin{proof}
Take $N=\mathrm{ord}(S^2)$.  We have from~\cite[Eq.(6)]{AEGN02} that
\[
\Delta(\g_F)=F\inv(\g_F\ox \g_F)(S^2\ox S^2)(F)
\]
(see also~\cite{majid95}).  Hence
\[
\Delta(\g_F^{(n)})=F\inv(\g_F^{(n)}\ox \g_F^{(n)})(S^{2n}\ox S^{2n})(F)
\]
for each $n$ and therefore
\[
\Delta_F(\g_F^{(N)})=F\Delta(\g_F^{(N)})F\inv=\g_F^{(N)}\ox \g_F^{(N)}.\qedhere
\]
\end{proof}

We have the following concrete description of the (universal) pivotalization of an equivalence $F:\mathscr{C}\to \mathscr{C}_F$ induced by a twist $F$ on $H$.

\begin{lemma}\label{lem:Fpiv}
The functor $F^{\piv}:\CCp\to\CCp_F$ sends an object $(V,\phi_V)$ in $\CCp$ to the pair consisting of the object $V$ along with the isomorphism
\[
V\to V^{**},\ \ v\mapsto \jj(\g_F\underline{\phi_V}(v)).
\]
In particular, $F^\piv(\RH)=(H^F, \jj\circ l(\g_F)\circ S^2)$.
\end{lemma}

\begin{proof}
Take $\beta=\beta_F$, $\g=\g_F$ and $\xi=\xi^F$.  Recall that $F(V^\ast)=F(V)^\ast=V^\ast$ as vector spaces for each $V$ in $\mathscr{C}$.  It follows from \eqref{eq:dual_tran2} that, for any object $V$ in $\mathscr{C}$,  $\tilde\xi: F(V^*) \to F(V)^{\ast}$ is given by
\[
\tilde{\xi}(f)=f\leftharpoonup \b\inv\ \ \text{for  }f\in V^*.
\]
 This implies
 $$
 \tilde \xi (\hat v)(f)  = (\hat v \leftharpoonup \b\inv)(f) = \hat v(\b\inv \cdot f)=f(S(\b\inv) v)=\jj(S(\b\inv) v) (f)
 $$
 for $\hat v \in F(V^{**})$ and $f \in F(V^*)$. Thus,
 $$
 \begin{aligned}
   ({\tilde\xi}^{\ast})\inv \tilde \xi (\hat v)(f) & = ({\tilde\xi}^{\ast})\inv \jj (S(\b\inv) v)(f)=\jj (S(\b\inv) v)(\tilde \xi\inv  (f)) \\
    &= \jj (S(\b\inv) v)(f \leftharpoonup \b) = f(\b S(\b\inv) v) = f(\g v) = \jj(\g v)(f)
 \end{aligned}
 $$
 for $\hat v \in F(V^{**})$ and $f \in F(V)^{\ast}$. By Theorem \ref{thm:piv_form}, $F^\piv(V, \phi_V) = (V,  (\tilde{\xi}^{\ast})\inv \tilde\xi \phi_V)$ and
 $$
 ({\tilde\xi}^{\ast})\inv \tilde\xi \phi_V(v) = ({\tilde\xi}^{\ast})\inv \tilde\xi \jj\underline{\phi_V}(v) = \jj(\g\underline{\phi_V}(v))
 $$
 for $v \in V$. The last statement follows immediately from the definition of $\RH = (H, \jj\circ S^2)$. This completes the proof.
\end{proof}

In the following proposition we let $\ol{S}^2$ denote the automorphism of $H/J(H)$ induced by $S^2$.

\begin{proposition}\label{prop:preseq}
Let $F\in H\ox H$ be a twist. The following statements are equivalent.
 \begin{enumerate}
   \item[(i)] $F^\upiv(\mathbf{R}_H)\cong \mathbf{R}_{H^F}$ in $\CC_F^{\upiv}$;
   \item[(ii)]  there is a unit $t$ in $H$ which satisfies the equation
\begin{equation}\label{eq:fundeq0}
S^2(t)\g_F^{-1}-t=0;
\end{equation}
   \item[(iii)] there is a unit $\ol t$ in $H/J(H)$ which satisfies the equation
\begin{equation}\label{eq:fundeq}
\bar{S}^2(\ol t)\ol \g_F^{-1}-\ol t=0\,.
\end{equation}
 \end{enumerate}
\end{proposition}

\begin{proof}
We take $N=\ord(S^2)$. By Lemma \ref{lem:Fpiv}, $F^{\piv}(\RH)=(H^F, \jj \circ l(\g_F)\circ S^2)$.  An isomorphism $F^{\piv}(\RH)\cong \mathbf{R}_{H^F}$ is determined by a $H^F$-module automorphism of $H^F$, which is necessarily given by right multiplication by a unit $t\in H^F$, producing a diagram
\[
\xymatrix{
H^F\ar[d]_{r(t)}\ar[rr]^-{l(\g_F)S^2} && H^F \ar[d]^-{r(t)} \ar[rr]^-{\jj} && (H^F)^{**} \ar[d]^-{r(t)^{**}}\\
{H^F} \ar[rr]_{S_F^2} && H^F \ar[rr]^{\jj} && (H^F)^{**}.}
\]

Equivalently, we are looking for a unit $t$ such that
\[
\g_FS^2(h)t = S_F^2(ht) = \g_FS^2(h)S^2(t)\g_F^{-1}
\]
for all $h \in H$. This equation is equivalent to
\begin{equation}
 S^2(t)\g_F^{-1}=t.
\end{equation}
\par

Let $\sigma$ denote the $\k$-linear automorphism $r(\g_F^{-1})\circ S^2=r(\g_F)^{-1}\circ S^2$ of $H^F$, and $\Sigma$ be the subgroup generated by $\sigma$ in $\Aut_\k(H^F)$.  Then we have $\sigma^N=r(\g_F^{(N)})^{-1}\circ S^{2N}=r(\g_F^{(N)})^{-1}$.  Since $\g_F^{(N)}$ is grouplike in $H^F$, it has a finite order.  Therefore $\sigma^N$ has finite order, as does $\sigma$, and $\Sigma$ is a finite cyclic group.
\par

Since $J(H)$ is a $\s$-invariant,
the exact sequence
$$0\to J(H)\to H\to H/J(H)\to 0$$
is in $\Rep(\Sigma)$.  Applying the exact functor $(-)^{\Sigma}$,  we get another exact sequence
\begin{equation}\label{eq:761}
0\to J(H)^{\Sigma}\to H^{\Sigma}\to (H/J(H))^{\Sigma}\to 0.
\end{equation}
\par

Recall that an element in $H$ is a unit if, and only if, its image in $H/J(H)$ is a unit.  So from the exact sequence~\eqref{eq:761}, we conclude that there is a unit in $(H/J(H))^{\Sigma}$ if and only if there is a unit in $H^{\Sigma}$.  Rather, there exists a unit $\bar{t}$ solving the equation $\sigma\cdot X-X=0$ in $H/J(H)$ if, and only if, there exists a unit $t$ solving the equation in $H$.  Since $\sigma\cdot \bar{t}=\bar{S}^2(\bar{t})\ol{\g}_F^{-1}$ and $\sigma\cdot t=S^2(t)\g_F^{-1}$,  the equation $\bar{S}^2(X)\ol{\g}_F^{-1}-X=0$ has a unit solution in $\ol{H}$ if, and only if, the equation $S^2(X)\g_F^{-1}-X=0$ has a unit solution in $H$.
\end{proof}

As an immediate consequence of this proposition, we can prove preservation of regular object for Hopf algebras with the Chevalley property.

\begin{theorem}\label{thm:presreg}
Suppose $H$ and $K$ are gauge equivalent finite-dimensional Hopf algebras with the Chevalley property, and  $\FF: \C{H} \to \C{K}$ is a monoidal equivalence. Then we have $\FF^\piv(\RH) \cong \RK$ in $(\C{K})^\piv$.
\end{theorem}

\begin{proof}
In view of \cite[Thm. 2.2]{NS2}, it suffices to assume $K = H^F$ for some twist $F \in H \o H$, and that $\FF$ is the associated equivalence $F: \C{H} \to \C{H^F}$.  Let $S$ be the antipode of $H$. It follows from Lemma \ref{l:reduction} that $\ol \g_F = \ol 1$ and $\ol S^2 = \id$. Therefore, every unit $t \in H/J(H)$ satisfies $\ol S^2(t)\ol \g_F^{-1}-t=0$. The proof is then completed by Proposition \ref{prop:preseq}.
\end{proof}

{ As a corollary we recover Theorem \ref{thm:KMN} for Hopf algebras with the Chevalley property.}

\begin{corollary}[{cf.~\cite[Thm. 2.2]{KMN}}]
If $\FF:H\mod\to K\mod$ is a gauge equivalence and $H$ has the Chevalley property then we have
\[
\nu^{\mathrm{KMN}}_n(H)=\nu^{\mathrm{KMN}}_n(K)
\]
for all $n\geq 0$.
\end{corollary}

\begin{proof}
We have $\FF^{\piv}(\mathbf{R}_H)\cong \mathbf{R}_K$ by Theorem~\ref{thm:presreg}.  Since a gauge equivalence preserves duals  this implies $\FF^{\piv}(\mathbf{R}_H^\ast)\cong \mathbf{R}_K^\ast$ as well.  Hence, using~\cite[Thm. 5.3 \& 5.7]{shimizu15}, we have
\[
\nu^{\mathrm{KMN}}_n(H)=\nu_n^{\mathrm{Sh}}(\mathbf{R}^\ast_H)=\nu_n^{\mathrm{Sh}}(\mathbf{R}^\ast_K)=\nu^{\mathrm{KMN}}_n(K).
\qedhere
\]
\end{proof}

\appendix

\section{Functoriality of finite pivotalizations}

We adopt the notation introduced at the beginning of Section~\ref{sect:Npiv}.  Recall that the subcategory $\CCNp\subset \CCp$ is the full subcategory consisting of all pairs $(V,\phi_V)$ such that the associated automorphism $\underline{\phi_V}\in\Aut_\k(V)$ satisfies $\mathrm{ord}(\underline{\phi_V})|N$.

\begin{lemma}\label{lem:preservers}
Let $F\in H\ox H$ be a twists and consider the functor $F:\CC\to \CC_F$.  Then, for any $N$ divisible by $\ord(S^2)$, the following statements are equivalent:
\begin{enumerate}
\item[(i)] $F^{\upiv}$ restricts to an equivalence $F^{\Npiv}:\CCNp\to \CC_F^{\Npiv}$.
\item[(ii)] $\g^{(N)}_F=1$.
\end{enumerate}
Furthermore, the existence of an isomorphism $F^{\piv}(\mathbf{R}_H)\cong \mathbf{R}_{H^F}$ implies {\rm (i)} and {\rm (ii)} for all such $N$.
\end{lemma}

\begin{proof}
Consider any $(V,\phi_V)$ in $\mathscr{C}^{\Npiv}$.  We have $F^{\piv}(V,\phi_V)=(V,\jj\circ l(\gamma_F)\circ \underline{\phi_V})$, by Lemma~\ref{lem:Fpiv}.  So $\underline{\phi_{F^{\piv}(V,\phi_V)}}=l(\gamma_F)\circ \underline{\phi_V}$.  Since $\underline{\phi_V}$, considered as an $H$-module map, is a map from $V$ to $\FF_{S^2}(V)$, we find by induction that
\[
(l(\gamma_F)\circ\underline{\phi_V})^n=l(\gamma_F^{(n)})\circ \underline{\phi_V}^n
\]
for each $n$.  In particular,
\begin{equation}\label{eq:988}
(l(\gamma_F)\circ \underline{\phi_V})^N=l(\gamma_F^{(N)})
\end{equation}
since $\underline{\phi_V}^N=1$.
\par

From equation~\eqref{eq:988} we see that $F^{\piv}(V,\phi_V)$ lay in $\mathscr{C}_F^{\Npiv}$ if and only if $l(\gamma_F^{(N)})=\id_V$.  Whence we have the implication (ii)$\Rightarrow$(i).  Applying \eqref{eq:988} to the case $(V,\phi_V)=\mathbf{R}_H$ gives the converse implication (i)$\Rightarrow$(ii) as well as the implication $F^{\piv}(\mathbf{R}_H)\cong \mathbf{R}_{H^F}\Rightarrow$(ii), since $\mathbf{R}_{H^F}$ is in each $\mathscr{C}_F^{\Npiv}$.
\end{proof}

We can now give the

\begin{proof}[Proof of Proposition~\ref{prop:Npiv}]
{In view of \cite[Thm. 2.2]{NS2}, it suffices to assume $K = H^F$ for some twist $F \in H \o H$ and consider the monoidal equivalence $F: \C{H} \to \C{H^F}$.
\par

For Hopf algebras with the Chevalley property:  Recall $\mathrm{ord}(S^2)=\mathrm{ord}(S^2_F)$ by Corollary~\ref{cor:ordpres}.  So we can pivotalize both $H$ and $H^F$ with respect to any $N$ divisible by $\mathrm{ord}(S^2)$.  We have already seen that $F^{\piv}(\mathbf{R}_H)\cong\mathbf{R}_{H^F}$.  It follows, by Lemma~\ref{lem:preservers}, that $F^{\piv}$ restricts to an equivalence $F^{\Npiv}:\mathscr{C}^{\Npiv}\to \mathscr{C}_F^{\Npiv}$.
\par

The general case:  From~\cite[Prop. 3.2]{EGqexp02} and the proof of~\cite[Prop. 3.3]{EGqexp02} we have $\gamma_F^{(\qexp(H))}=1$.  By Lemma \ref{lem:preservers} it follows that $F^\piv$ restricts to an equivalence $F^{\Epiv}:\CCEp\to \CC_F^{\Epiv}$.}
\end{proof}
\subsection*{Acknowledgements}
The authors would like to thank the referee for his careful reading and suggestions on the previous version of this paper.
\bibliographystyle{abbrv}

\end{document}